\documentclass[12pt]{amsart}
\usepackage{amssymb}

\usepackage{caption}
\usepackage{amsmath, amssymb, mathrsfs, verbatim, multirow}
\usepackage[all]{xy}
\usepackage{pifont}
\usepackage{imakeidx}
\usepackage{float}
\usepackage{color}
\usepackage{graphicx,color,epsfig}
\usepackage{pgf,tikz}
\usepackage{pgfplots}
\pgfplotsset{compat=1.15}
\usepackage{mathrsfs}
\usetikzlibrary{patterns}
\usepackage{subfigure}
\usepackage[shortlabels]{enumitem} 
\usepackage{tkz-tab}

\usepackage{subfigure}

\usepackage{color}
\newtheorem{theorem}{Theorem}[section]
\newtheorem{lemma}[theorem]{Lemma}
\newtheorem{remark}[theorem]{Remark}

\newtheorem{corollary}[theorem]{Corollary}
\newtheorem{proposition}[theorem]{Proposition}

\newtheorem{lem-def}[theorem]{Lemma-Definition}

\renewenvironment{proof}{{\bfseries Proof.}}{\qed}
\newcommand{\R}{\mathbb R}

\newcommand{\N}{\mathbb N}

\newcommand{\Q}{\mathbb Q}

\def\ars#1{\renewcommand\arraystretch{#1}}

\def\diso{\lower.4ex\hbox{$\downarrow$}\raise.4ex\hbox{\mbox{\scriptsize
$\wr$}}}

\def\ism{\lower.3ex\hbox{\ars{.08}$\begin{array}{c}\,\to\\\mbox{\tiny $\sim\,$}\end{array}$}}
\def\iso{\lower.3ex\hbox{\ars{.08}$\begin{array}{c}\lra\\\mbox{\tiny $\sim\,$}\end{array}$}}

\def\lg{l\raise.6ex\hbox to.2em{\hss.\hss}l}

\def\lra{\,\longrightarrow\,}

\def\orb{\hbox to  .3em{$\backslash$}\backslash}

\DeclareMathOperator{\inv}{in}
\newcounter{cs}
\stepcounter{cs}
\newcommand{\casos}{\begin{itemize}}
\newcommand{\fcasos}{\end{itemize}\setcounter{cs}{1}}

\newfont{\tit}{cmr12 scaled \magstep3}

\author{Josnei Novacoski}
\address{Departamento de Matem\'{a}tica,         Universidade Federal de S\~ao Carlos, Rod. Washington Luís, 235, 13565--905, S\~ao Carlos -SP, Brazil}
\email{josnei@ufscar.br}
\thanks{During the realization of this project the author was supported by two grants from Funda\c{c}\~ao de Amparo \`a Pesquisa do Estado de S\~ao Paulo (process numbers 2017/17835-9 and 2021/11246-7).}

\title[Characterization of defect]{A characterization for the defect of rank one valued field extensions}

\begin{document}
\begin{abstract}
In this paper we present a characterization for the defect of a simple algebraic extension of rank one valued fields using the key polynomials that define the valuation. As a particular example, this gives the classification of defect extensions of degree $p$ as dependent or independent presented by Kuhlmann.
\end{abstract}
\maketitle
\section{Introduction}
Let $(L/K,v)$ be a finite valued field extension. Suppose that $L=K(\eta)$ for some $\eta\in L$ and let $g$ be the minimal polynomial of $\eta$ over $K$. We will consider the valuation $\nu$ on $K[x]$ with support $gK[x]$ defined by $v$ and $\eta$. Namely, for any $f\in K[x]$ we set $\nu(f):=v(f(\eta))$.

Fix an extension $\overline \nu$ of $\nu$ to $\overline K[x]$, where $\overline K$ is a fixed algebraic closure of $K$. For each $f\in K[x]$ we define
\[
\epsilon(f):=\max\{\overline \nu(x-a)\mid a\mbox{ is a root of }f\}.
\]
By \cite[Remark 3.2]{NovaMP} the value $\epsilon(f)$ does not depend on the extension $\overline \nu$ of $\nu$. A monic polynomial $Q\in K[x]$ is called \textbf{a key polynomial for $\nu$} if
\[
\deg(f)<\deg(Q)\Longrightarrow \epsilon(f)<\epsilon(Q)\mbox{  for all }f\in K[x].
\]

Let $vL$ be the value group of $v$ and denote by $\Gamma$ the divisible closure of $vL$. For $n\in \N$ we denote by $\Psi_n$ the set of all the key polynomials for $\nu$ of degree $n$. We will say that $\Psi_n$ does not have a maximum or that $\Psi_n$ is bounded in $\Gamma$ if the same property is satisfied for $\nu(\Psi_n)$. A \textbf{key polynomial for $\Psi_n$} is a key polynomial for $\nu$ of smallest degree larger than $n$. We denote by ${\rm KP}(\Psi_n)$ the set of all the key polynomials for $\Psi_n$. If $\Psi_n$ does not have a maximum, then any key polynomial for $\Psi_n$ will be called a \textbf{limit key polynomial for $\Psi_n$}. In this case, we say that $\Psi_n$ is a \textbf{plateau} for $\nu$.

For polynomials $f,q\in K[x]$, $q\notin K$, there exist uniquely determined $f_0,\ldots,f_r\in K[x]$ with $\deg(f_i)<\deg(q)$ for every $i$, $0\leq i\leq r$, such that
\[
f=f_0+f_1q+\ldots+f_rq^r.
\]
This expression is called the \textbf{$q$-expansion of $f$}. For any key polynomial $Q$ for $\nu$ and $f\in K[x]$ we will denote by
\[
f=a_{Q0}(f)+a_{Q1}(f)Q+\ldots+a_{Qr}(f)Q^r
\]
the $Q$-expansion of $f$. We set
\[
L_Q(f)=\{i\in\N_0\mid a_{Qi}(f)\neq 0\} 
\]
i.e., the set of indices of the non-zero monomials in the $Q$-expansion of $f$. We define the \textbf{truncation of $\nu$ at $Q$} as
\[
\nu_Q(f)=\min_{i\in L_Q(f)}\{\nu\left(a_{Qi}(f)Q^i\right)\}.
\]
This mapping is a valuation (\cite[Proposition 2.6]{SopivNova}).

For $n\in \N$, $n<\deg(g)$, such that $\Psi_n\neq \emptyset$ the fact that $\nu(g)=\infty$ implies that $\Psi_n$ admits a key polynomial $F$. In particular, if $Q\in \Psi_n$, then $\nu_Q\rightarrow\nu_{F}$ is an augmentation (\cite[Theorems 6.1 and 6.2]{Nov1}). Moreover, this augmentation is a limit augmentation if and only if $\Psi_n$ does not have a maximum. Hence, we can define the \textbf{defect $d(\Psi_n)$} of $\Psi_n$ by $d(\nu_Q\rightarrow \nu_{F}$) (see more details in Section \ref{definitiondeffect}).   

We will denote by $p$ the characteristic exponent of $Kv$. The main goal of this paper is to prove the following result. 
\begin{theorem}\label{mainthm}
Assume that $d(L/K,v)=p^d$ and that ${\rm rk}(v)=1$. Then there exist uniquely determined $d_1,\ldots,d_{r-1}\in\N$, $d_r\in\N_0$, and for every $i$, $1\leq i< r$, a uniquely determined subset $I_i\subseteq \{0,\ldots,d_i-1\}$ such that the following hold.
\begin{description}
\item[(i)] $d=d_1+\ldots+d_r$.
\item[(ii)] There exist $n_1,\ldots,n_r\in \N$ with $n_1<n_2<\ldots<n_r$ such that $\Psi_{n_i}$, $1\leq i\leq r$, are all the plateaus for $\nu$.
\item[(iii)] For every $i$, $1\leq i\leq r$, we have $d(\Psi_{n_i})=p^{d_i}$.
\end{description}
For each $i$, $1\leq i< r$, and every limit key polynomial $F$ for $\Psi_{n_i}$, there exists $Q_i\in \Psi_{n_i}$ such that for every $Q\in \Psi_{n_i}$ with $\nu(Q)\geq \nu(Q_i)$ we have:
\begin{description}
\item[(iv)]
\[
p^{I_i}:=\{p^j\mid j\in I_i\}\subseteq L_Q(F)\mbox{; and}
\] 
\item[(v)]
\begin{equation}\label{equalimiktp}
a_{Q0}(F)+\sum_{j\in I_i}a_{Qp^j}(F)Q^{p^j}+Q^{p^{d_i}}
\end{equation}
is a limit key polynomial for $\Psi_{n_i}$.
\end{description}
Moreover, if $\Psi_{n_r}$ is bounded in $\Gamma$, then we can also find a uniquely determined $I_{r}$, and for $F\in {\rm KP}(\Psi_{n_r})$ a polynomial $Q_r\in\Psi_{n_r}$, satisfying \textbf{(iv)} and \textbf{(v)} (for $i=r$). 
\end{theorem}

Theorem \ref{mainthm} can be seen as a generalization of the classification of defect extensions of degree $p$ presented by Kuhlmann in \cite{Kuhl} and extended by Kuhlmann and Rzepka in \cite{KR}. For a subset $S\subseteq \Gamma\cup\{\infty\}$, we define $\overline S$ as the cut  on $\Gamma$ having the lower cut set given by
\[
\{\gamma\in \Gamma\mid\ \exists s\in S\mbox{ with }\gamma\leq s\}.
\]
Also, we define $S^-$ as the cut on $\Gamma$ having the lower cut set given by
\[
\{\gamma\in \Gamma\mid \gamma<s\mbox{ for every }s\in S\}.
\]
Suppose that $vL=vK$. The \textbf{distance of $\eta$ to $K$} is the cut
\[
{\rm dist}(\eta,K)=\overline{\{v(\eta-b)\mid b\in K\}}.
\]

In \cite{Kuhl} and \cite{KR}, the authors consider independent and dependent defect extensions in two cases. We will say that we are in the \textbf{Artin-Schreier case} if
\begin{equation}                           \label{sitAS}
\left\{\begin{array}{l}
L=K(\eta)\mbox{ is an Artin-Schreier extension of }K,\\
\mbox{the minimal polynomial of }\eta\mbox{ over }K\mbox{ is }g=x^p-x-a; and\\
d(L/K,v)=p
\end{array}\right. .
\end{equation}
We will say that we are in the \textbf{Kummer case} if
\begin{equation}                           \label{sitkummer}
\left\{\begin{array}{l}
K\mbox{ contains a }p\mbox{-th root of unity}\\
L=K(\eta)\mbox{ is a Kummer extension of }K,\\
\mbox{the minimal polynomial of }\eta\mbox{ over }K\mbox{ is }g=x^p-a; and\\
v(a)=0\mbox{ and }d(L/K,v)=p
\end{array}\right. .
\end{equation}

In the situation \eqref{sitAS} we say that $(L/K,v)$ is \textbf{independent} if
\[
{\rm dist}(\eta, K)=H^-\mbox{ for some convex subgroup }H\mbox{ of }\Gamma.
\]
Otherwise, it is called \textbf{dependent}. If \eqref{sitkummer} is satisfied, then we say that $(L/K,v)$ is \textbf{independent} if
\[
{\rm dist}(\eta, K)=\frac{v(p)}{p-1}+H^-\mbox{ for some convex subgroup }H\mbox{ of }\Gamma.
\]
Otherwise, it is called \textbf{dependent}.

\begin{proposition}\label{independengrankeone}
Assume that either \eqref{sitAS} or \eqref{sitkummer} is satisfied. Suppose that ${\rm rk}(v)=1$ and consider the valuation $\nu$ on $K[x]$ defined by $v$ and $\eta$. Then, in the notation of Theorem \ref{mainthm}, we have $r=1$ and $d_1=n_1=1$. Moreover, $\Psi_1$ is bounded in $\Gamma$ and 
\begin{equation}\label{equatiocharacindeprn}
I_1=\emptyset\mbox{ if and only if } (L/K,v) \mbox{ is dependent.}
\end{equation}
\end{proposition}
Since the only possibilities for $I_1$ are $\emptyset$ or $\{0\}$, the condition \eqref{equatiocharacindeprn} is equivalent to  
\[
I_1=\{0\}\mbox{ if and only if } (L/K,v) \mbox{ is independent.}
\]

The sets $I_i$, appearing in Theorem \ref{mainthm}, have a very explicit description. This description can be generalized even if ${\rm rk}(v)\neq 1$. Namely, for a plateau $\Psi_n$ and a limit key polynomial $F$ for $\Psi_n$ we consider the cut
\[
\delta_F=\overline{\{\nu_Q(F)\}_{Q\in \Psi_n}}
\]
on $\Gamma$. There exists $D\in\N$ such that for every $Q\in \Psi_n$ the $Q$-expansion of $F$ is of the form
\[
F=a_{Q0}(F)+a_{Q1}(F)Q+\ldots+a_{QD}(F)Q^D.
\]
We define
\[
B(F)=\{b\in \{1,\ldots,D-1\}\mid \nu\left(a_{Qb}(F)Q^b\right)\in \delta_F^L\mbox{ for every }Q\mbox{ with large enough value}\}.
\]
For a plateau $\Psi_{n_i}$ and a limit key polynomial $F$ for $\Psi_{n_i}$ as in Theorem \ref{mainthm}, the set $I_i$ will be defined as the numbers $j$ for which $p^j\in B(F)$. 

The next result is a generalization of Proposition \ref{independengrankeone} for rank greater than one.
\begin{proposition}\label{propdockuhlma}
Assume that either \eqref{sitAS} or \eqref{sitkummer} is satisfied. Consider the valuation $\nu$ on $K[x]$ defined by $v$ and $\eta$. Then $\delta_g< \infty^-$ and
\begin{equation}\label{indepenchara}
B(g)=\emptyset\Longleftrightarrow (L/K,v)\mbox{ is dependent.}
\end{equation}
\end{proposition}

\textbf{Acknowledgements.} I would like to thank Mark Spivakovsky for carefully reading, for providing useful suggestions and for pointing out a few mistakes in an earlier version of this paper. He also provided simpler and more complete arguments for some of the steps in the proofs. I also thank Enric Nart and the anonymous referee for useful comments and corrections.

\section{The defect of an augmentation}\label{definitiondeffect}
Let $\mu$ be a valuation on $K[x]$ with value group $\Gamma_\mu$. The \textbf{graded ring of $\mu$} is defined as
\[
\mathcal G_\mu:=\bigoplus_{\gamma\in \Gamma_\mu}\{f\in K[x]\mid \mu(f)\geq \gamma\}/\{f\in K[x]\mid \mu(f)> \gamma\}.
\]
For $h\in K[x]$ for which $\mu(h)\neq \infty$, we define the \textbf{initial form} of $h$ in $\mathcal G_\mu$ by
\[
\inv_\mu(h):=h+\{f\in K[x]\mid \mu(f)> \mu(h)\}\in \mathcal G_\mu.
\]

Let $(L/K,v)$ be a simple algebraic valued field extension (not necessarily of rank one) and take $\eta\in L$ such that $L=K(\eta)$. Consider the corresponding valuation $\nu$ on $K[x]$ defined by $v$ and $\eta$. For a key polynomial $Q$ for $\nu$ we can consider the graded ring of $\nu_Q$ which we denote by $\mathcal G_Q$ (instead of $\mathcal G_{\nu_Q}$). For $f\in K[x]$, with $\nu_Q(f)\neq \infty$, we denote $\inv_Q(f):=\inv_{\nu_Q}(f)$. Let
\[
R_Q:=\langle\{\inv_Q(f)\mid \deg(f)<\deg(Q)\}\rangle\mbox{ and }y_Q:=\inv_Q(Q)\in \mathcal G_Q.
\]
This means that $R_Q$ is the abelian subgroup of $\mathcal{G}_Q$ generated by the initial forms of polynomials of degree smaller than $\deg(Q)$. 
\begin{proposition}\cite[Proposition 4.5]{Nov1}
The set $R_Q$ is a subring of $\mathcal G_Q$, $y_Q$ is transcendental over $R_Q$ and
\[
\mathcal G_Q=R_Q[y_Q].
\]
\end{proposition}
In view of the previous proposition, for every $f\in K[x]$, with $\nu_Q(f)\neq \infty$, we can define the \textbf{degree of $f$ with respect to $Q$} as the degree of $\inv_Q(f)$ with respect to $y_Q$, i.e.,
\[
\deg_Q(f):=\deg_{y_Q}(\inv_Q(f)).
\]

For $n\in \N$, suppose that $\Psi_n$ does not have a maximum and that $\Psi_n$ admits a limit key polynomial $F$. For any $Q\in\Psi_n$, by \cite[Theorem 6.2]{Nov1}, $\nu_Q\rightarrow \nu_F$ is a limit augmentation. Hence, we can define the \textbf{defect of }$\Psi_n$ (denoted by $d(\Psi_n)$) as the defect of $\nu_Q\rightarrow \nu_F$ as in \cite[Definition 6.2]{NN}. Namely,
\[
d(\Psi_n):=\lim\limits_{\substack{Q'\in \Psi_n\\\nu(Q')\geq \nu(Q)}}\{\deg_{Q'}(F)\}.
\]
\begin{theorem}\label{importate}
Let $(L/K,v)$ be a simple algebraic valued field extension and fix a generator $\eta$ of $L|K$. Consider the corresponding valuation $\nu$ on $K[x]$ defined by $v$ and $\eta$. Let $n_1,\ldots,n_r\in\N$ be all the natural numbers $n$ for which $\Psi_n$ is a plateau. Then
\[
d(L/K,v)=\prod_{i=1}^r d(\Psi_{n_i}).
\]
Moreover, if ${\rm rk}(v)=1$, then for every $i$, $1\leq i\leq r$, for which $\Psi_{n_i}$ is bounded in $\Gamma$ we have
\begin{equation}\label{equationsdefectaug}
d(\Psi_{n_i})=\frac{\deg(F)}{\deg(Q)}\mbox{ for every }Q\in \Psi_{n_i}\mbox{ and every }F\in {\rm KP}(\Psi_{n_i}).
\end{equation}
\end{theorem}
\begin{proof}
Let $\{m_1,\ldots,m_s\}$ be the set of natural numbers $m$ for which $\Psi_m$ is non-empty. For $j$, $1\leq j\leq s$, if $\Psi_{m_j}$ has a maximum, then we choose $Q_j\in \Psi_{m_j}$ such that $\nu(Q_j)$ is the maximum. If $\Psi_{m_j}$ does not have a maximum (i.e., $m_j\in\{n_1,\ldots,n_r\}$), then we choose any $Q_j\in \Psi_{m_j}$. It follows from \cite[Theorems 6.1 and 6.2]{Nov1} that
\[
\nu_{Q_1}\rightarrow \nu_{Q_2}\rightarrow\ldots\rightarrow \nu_{Q_s}=\nu
\]
satisfies the conditions of \cite[Theorem 6.14]{NN}, and consequently
\[
d(L/K,v)=\prod_{j=1}^{s-1}d\left(\nu_{Q_j}\rightarrow \nu_{Q_{j+1}}\right)=\prod_{i=1}^rd(\Psi_{n_i}).
\]
The last equality holds because $d\left(\nu_{Q_j}\rightarrow \nu_{Q_{j+1}}\right)=1$ if the augmentation is ordinary (\cite[Lemma 6.3]{NN}) and by the definition of $d(\Psi_{n_i})$.

If ${\rm rk}(v)=1$, then \eqref{equationsdefectaug} follows from \cite[Corollary 7.7]{NN}.
\end{proof}
\begin{remark}\label{rmkdefecnoti}
Observe that as a consequence of \eqref{equationsdefectaug} we deduce that, in this case, $d(\Psi_{n_i})>1$.
\end{remark}

\section{The subset $I$}
For this section we assume that ${\rm rk}(v)=1$, so we can suppose that $vL\subseteq \R$. Fix $n\in \N$ for which $\Psi_n$ does not have a maximum and is bounded in $\Gamma$. Throughout this section, we will fix a limit key polynomial $F$ for $\Psi_n$. We denote
\[
K[x]_n:=\{f\in K[x]\mid \deg(f)<n\}.
\]

Write $F=L(Q)$ for some $L(X)\in K[x]_n[X]$ and denote $D:=\deg_X(L)$. The polynomial $L(X)$ depends on $Q$ and can be obtained from the $Q$-expansion of $F$. Namely,
\[
L(X)=a_{Q0}(F)+a_{Q1}(F)X+\ldots+a_{QD}(F)X^D.
\]
Set
\[
B=\lim_{Q\in \Psi_n}\nu(Q)\mbox{ and }\overline B=\lim_{Q\in \Psi_n}\nu_Q(F).
\]
Take $Q_0\in \Psi_n$ and choose $Q\in \Psi_n$ such that
\begin{equation}\label{eqaescolhaimpor}
\epsilon(Q)-\epsilon(Q_0)>D(B-\nu(Q)).
\end{equation}
By \cite[Proposition 2.10 \textbf{(iii)}]{SopivNova} we deduce that for $Q,Q'\in \Psi_n$ we have $\nu(Q')>\nu(Q)$ if and only if $\epsilon(Q')>\epsilon(Q)$. Hence, $Q\in \Psi_n$ as in \eqref{eqaescolhaimpor} exists because the left hand side grows and the right hand side tends to zero when $\nu(Q)$ grows.
\begin{remark}
We can deduce from \cite[Proposition 4.1]{Netal} that $\overline B= D\cdot B$ and $\nu_Q(F)=D\cdot \nu(Q)$ for $Q$ with large enough value. Hence, condition \eqref{eqaescolhaimpor} is equivalent to
\begin{equation}\label{eqaescolhaimpor2}
\epsilon(Q)-\epsilon(Q_0)>\overline B-\nu_Q(F).
\end{equation}
\end{remark}
We will consider the ring $K(x)[X]$ where $X$ is an indeterminate and let $\partial_i$ denote the $i$-th  Hasse derivative with respect to $X$. Then, for every $l(X)\in K(x)[X]$ and $a,b\in K(x)$ we have the Taylor expansion
\[
l(b)=l(a)+\sum_{i=1}^{\deg_Xl}\partial_il(a)(b-a)^i.
\]

For simplicity of notation, we will take a well-ordered family $\{Q_\rho\}_{\rho<\lambda}$ whose values $\gamma_\rho:=\nu(Q_\rho)$ are larger than $\nu(Q)$ and form a cofinal family in $\nu(\Psi_n)$. For each $\rho<\lambda$, set $h_\rho=Q-Q_\rho$. In particular, we can consider the Taylor expansion of $F$ with respect to $h_\rho$:
\begin{equation}\label{taylorexpan}
F=L(h_\rho)+\sum_{i=1}^D\partial_{i}L(h_\rho)Q_\rho^{i}.
\end{equation}
For simplicity of notation we will denote $\nu_\rho:=\nu_{Q_\rho}$ for every $\rho<\lambda$.
\begin{lemma}\cite[Corollary 2.5]{NOVSpiv}\label{Corolarcurucial}
If $\deg(f)< \deg(F)$, $f=l(Q)$ for some $l(X)\in K[x]_n[X]$, then there exists $\rho$ such that
\[
\nu(l(h_\sigma))=\nu(f)=\nu_\sigma(f)=\nu(a_{\sigma 0}(f))
\]
for every $\sigma$, $\rho< \sigma<\lambda$.
\end{lemma}
For each $i$, $1\leq i\leq D$, the polynomial $\partial_iL(Q)$ has degree smaller than $\deg(F)$. Hence, by Lemma \ref{Corolarcurucial} there exists $\rho_0<\lambda$ such that
\begin{equation}\label{equanajudamuito}
\beta_i:=\nu(\partial_iL(Q))=\nu(\partial_iL(h_{\rho}))\mbox{ for every }\rho, \rho_0\leq\rho<\lambda.
\end{equation}
Moreover, by \cite[Lemma 4]{Kap}, we can take $\rho_0$ so large that for every $j,k$, $1\leq j<k \leq D$, we have
\begin{equation}\label{equalelineslope}
\beta_j+j\gamma_\rho\neq\beta_k+k\gamma_\rho\mbox{ for every }\rho_0\leq \rho<\lambda.
\end{equation}
From now on, we will only consider $\rho$ (and consequently $\sigma$ and $\theta$ appearing below), such that \eqref{equanajudamuito} and \eqref{equalelineslope} are always satisfied (i.e., $\min\{\rho,\sigma,\theta\}>\rho_0$). 

For each $\rho<\lambda$ denote
\[
F=a_{\rho 0}(F)+a_{\rho 1}(F)Q_\rho+\ldots+a_{\rho D}(F)Q_\rho^D
\]
the $Q_\rho$-expansion of $F$.

\begin{lemma}\cite[Lemma 4.2]{NOVSpiv}\label{Lemaparacabcomtuo}
Fix $\rho<\lambda$ and for each $i$, $0\leq i\leq D$, set $b_{\rho i}:=a_{\rho 0}(\partial_iL(h_\rho))$. Then
\[
\nu_\rho(\partial_iL(h_\rho)-b_{\rho i})+i\nu(Q_\rho)>\overline B.
\]
\end{lemma}
\begin{corollary}\label{corolallryuniquenes}
With the notation above, we have
\[
\nu\left(a_{\rho i}(F)Q_\rho^i\right)>\overline B\Longleftrightarrow \beta_i+i\gamma_\rho>\overline B.
\]
\end{corollary}
\begin{proof}
By the Taylor expansion of $F$ with respect to $h_\rho$, we have
\[
a_{\rho i}(F)=b_{\rho i}+G
\]
where
\[
G=\sum_{j< i}a_{\rho i-j}\left(\partial_jL(h_\rho)\right)=\sum_{j<i}a_{\rho i-j}\left(\partial_jL(h_\rho)-b_{\rho j}\right).
\]
Hence, the result follows trivially from Lemma \ref{Lemaparacabcomtuo}.
\end{proof}\medskip

Denote by $J_\rho(F)$ the set
\[
J_\rho(F)=\{j\in\{1,\ldots,D\}\mid \nu\left(a_{\rho j}(F)Q_\rho^j\right)>\overline B\}.
\]

\begin{corollary}\label{corolalinelimit}
For $i\notin J_\rho(F)$ we have
\[
\nu(a_{\rho i}(F))=\beta_i.
\] 
\end{corollary}
\begin{proof}
It follows again from the definition of $J_\rho(F)$ and Lemma \ref{Lemaparacabcomtuo}.
\end{proof}\medskip

\begin{corollary}
If $\rho<\sigma$, then $J_\rho(F)\subseteq J_\sigma(F)$.
\end{corollary}
\begin{proof}
Follows trivially from Corollary \ref{corolallryuniquenes}.
\end{proof}\medskip

Since $\{1,\ldots,D\}$ is finite, there exists $\rho_1$, $\rho_0\leq \rho_1<\lambda$ such that for every $\rho$, $\rho_1\leq \rho<\lambda$ we have $J_\rho(F)= J_{\rho_1}(F)$. Set
\[
B_n(F)=\{1,\ldots,D-1\}\setminus J_{\rho_1}(F).
\] 

For a subset $S$ of $\{1,\ldots,D-1\}$ and $\rho$, $\rho_1\leq \rho<\lambda$ we denote by
\[
F_{S,\rho}=a_{\rho 0}(F)+\sum_{s\in S}a_{\rho s}(F)Q_\rho^s+a_{\rho D}Q_\rho^D.
\]
\begin{proposition}\label{existenceunuequenusI}
Take $\rho<\lambda$, $\rho_1\leq \rho<\lambda$, and $S\subseteq B_n(F)$. Then $F_{S,\rho}$ is a limit key polynomial for $\Psi_n$ if and only if $S=B_n(F)$.
\end{proposition}
\begin{proof}
Suppose that $S=B_n(F)$. Then for every $Q\in \Psi_n$, with $\nu(Q)\geq \nu(Q_\rho)$ we have
\[
\nu_Q\left(F-F_{S,\rho}\right)=\min_{j\in J_{\rho}(F)}\left\{\nu_Q\left(a_{\rho j}(F)Q_\rho^j\right)\right\}=\min_{j\in J_{\rho}(F)}\left\{\nu\left(a_{\rho j}(F)Q_\rho^j\right)\right\}>\overline B>\nu_Q(F). 
\]
Hence, $\nu_Q(F)=\nu_Q(F_{S,\rho})$. Since $\deg(F)=\deg(F_{S,\rho})$ we conclude that $F_{S,\rho}$ is also a limit key polynomial for $\Psi_n$. 

Suppose now that $S\subsetneq B_n(F)$. For any $Q\in \Psi_n$ such that
\[
\nu_Q(F)>\beta_h+h\gamma_\rho:=\min_{k\in B_n(F)\setminus S}\{\beta_k+k\gamma_\rho\}
\]
we have (by \eqref{equalelineslope})
\[
\nu_Q(F_{S,\rho}-F_{B_n(F),\rho})=\beta_h+h\gamma_\rho<\nu_Q(F)=\nu_Q(F_{B_n(F),\rho}).
\]
Hence, $\nu_Q(F_{S,\rho})=\beta_h+h\gamma_\rho$, which implies that $F_{S,\rho}$ is not a limit key polynomial for $\Psi_n$.
\end{proof}

\begin{proposition}
Let $H$ be another limit key polynomial for $\Psi_n$. Then $B_n(F)=B_n(H)$.
\end{proposition}
\begin{proof}
Since both $F$ and $H$ are monic, the polynomial $h=H-F$ has degree smaller than $\deg(F)$. Hence, there exists $\theta<\lambda$ such that $\nu_\sigma(h)=\nu_\theta(h)$ for every $\sigma$, $\theta\leq \sigma<\lambda$. Since $\{\nu_\rho(F)\}_{\rho<\lambda}$ and $\{\nu_\rho(H)\}_{\rho<\lambda}$ are increasing, this implies that
\begin{equation}\label{equationsigma}
\nu_\sigma(h)\geq \overline B\mbox{ and }\nu_\sigma(F)=\nu_\sigma(H)\mbox{ for every }\sigma,\ \theta\leq\sigma<\lambda.
\end{equation}

Take $j\in B_n(H)$. This means that for every $\sigma$, $\theta<\sigma<\lambda$, we have $\nu\left(a_{\sigma j}(H)Q_\sigma^j\right)< \overline B$. Since $a_{\sigma j}(F)=a_{\sigma j}(H)+a_{\sigma j}(h)$, this and \eqref{equationsigma} imply that
\[
\nu\left(a_{\sigma j}(F)Q_\sigma^j\right)=\nu\left(a_{\sigma j}(H)Q_\sigma^j\right)< \overline{B}.
\]
Hence $B_n(H)\subseteq B_n(F)$. The other inclusion follows by the symmetric argument.
\end{proof}\medskip

Since the set $B_n(F)$ does not depend on the choice of $F$, we will denote it  by $B_n$.

When referring to a polynomial $Q\in \Psi_n$ with large enough value we mean that
\begin{equation}\label{menainglarge}
Q\mbox{ satisfies }\eqref{eqaescolhaimpor}, \ \nu(Q)>\nu(Q_{\rho_0})\mbox{ and }\nu(Q)>\nu(Q_{\rho_1}).
\end{equation}
\begin{corollary}\label{corolastablecoeff}
For every key polynomial $F$ for $\Psi_n$ and every $\sigma$ for which $Q_\sigma\in \Psi_n$ satisfies \eqref{menainglarge} we have
\[
B_n\subseteq L_{Q_\sigma}(F).
\]
\end{corollary}
\begin{proof}
Take $i\in B_n$. By Corollary \ref{corolallryuniquenes} and Corollary \ref{corolalinelimit} we have
\[
\nu\left(a_{\sigma i}(F)Q_\sigma^i\right)=\beta_i+i\gamma_\sigma<\overline{B}.
\]
In particular, $i\in L_{Q_\sigma}(F)$.
\end{proof}

\subsection{Geometric interpretation of $B_n$}\label{geometric}
For $F\in {\rm KP}(\Psi_n)$ and $Q\in \Psi_n$ we denote by $\Delta_Q(F)$ the \textbf{Newton polygon of $F$ with respect to $Q$}. This is defined as the lowest part of the convex hull of 
\[
\{(i,\nu(a_{Qi}(F)))\mid i\in \N_0\}
\]
in $\Q\times \Gamma$. If $\Psi_n$ is bounded, then for large enough $Q$ the set $\Delta_Q(F)$ is the line segment connecting $(p^d,0)$ and $(0,p^d\nu(Q))$ for $p^d=d(\Psi_n)$ (by \cite[Proposition 3.2 and Lemma 4.2]{NOVSpiv}). Consider the line $\pi$ passing through $(p^d,0)$ and $(0,\overline B)$. Since $\overline B=p^dB$, this is the line with equation
\[
\pi(y)=-By+\overline B.
\]
\begin{lemma}
For $k\in \{1,\ldots, D-1\}$ we have $k\in B_n$ if and only if $(k,\beta_k)\in \pi$. 
\end{lemma}
\begin{proof}
By \cite[Proposition 3.2 and Lemma 4.2]{NOVSpiv}, for every $\rho<\lambda$ with large enough value we have $p^d\gamma_\rho\leq \beta_k+k\gamma_\rho$. Hence, $k\in B_n$ if and only if
\[
p^d\gamma_\rho\leq\beta_k+k\gamma_\rho<\overline{B}=p^d B\mbox{ for every }\rho<\lambda. 
\]
Taking the supremum of each of the expressions, this is equivalent to
\[
p^d B\leq \beta_k+kB\leq p^dB.
\]
This is equivalent to $\beta_k=-Bk+\overline B$ and this happens if and only if $(k,\beta_k)\in \pi$.
\end{proof}\medskip

In Figure \ref{figuranesmple} below we present the characterization of the set $B_n$ using Newton polygons describe above. We consider $Q\in\Psi_n$ with large enough value and $F\in{\rm KP}(\Psi_n)$. The Newton polygon $\Delta_Q(F)$ is represented in blue. The blue dots represent the points $(i,\nu(a_{Qi}(F)))$. The line $\pi$ is represented in red.
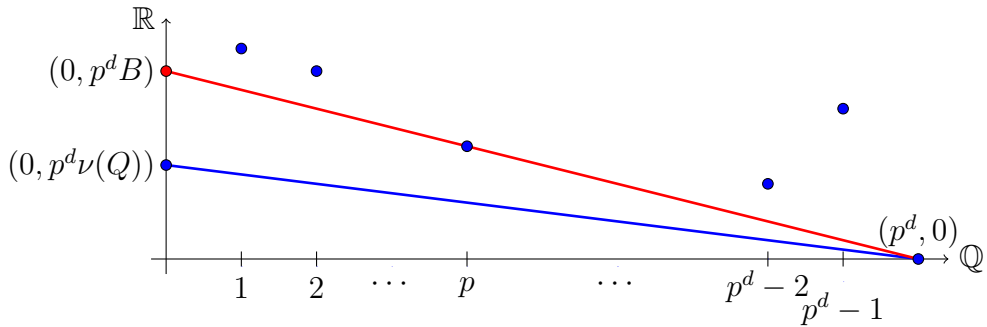
\begin{figure}[H]
    \centering
\begin{tikzpicture}
    \draw[->] (-0.2,0) -- (10.4,0) node[right] {$\Q$};
    \draw[->] (0,-0.2) -- (0,3.2) node[left] {$\R$};
    
    \draw[-] (1,-0.1) -- (1,0.1);
    \draw[-] (2,-0.1) -- (2,0.1);
    \draw[-] (4,-0.1) -- (4,0.1);
    \draw[-] (8,-0.1) -- (8,0.1);
    \draw[-] (9,-0.1) -- (9,0.1);

  \draw [line width=1.pt,color=blue] (10,0)-- (0,1.25);
  \draw [line width=1.pt,color=red] (10,0)-- (0,2.5);

\filldraw[fill=blue] (1,-0.1) circle (0pt) node[below] {$1$};
\filldraw[fill=blue] (2,-0.1) circle (0pt) node[below] {$2$};
\filldraw[fill=blue] (3,-0.1) circle (0pt) node[below] {$\cdots$};
\filldraw[fill=blue] (4,-0.1) circle (0pt) node[below] {$p$};
\filldraw[fill=blue] (6,-0.1) circle (0pt) node[below] {$\cdots$};
\filldraw[fill=blue] (8,0) circle (0pt) node[below] {$p^d-2$};
\filldraw[fill=blue] (9,-0.3	) circle (0pt) node[below] {$p^d-1$};

   \filldraw[fill=red] (0,2.5) circle (2pt) node[left] {$(0,p^dB)$};
  \filldraw[fill=blue] (10,0) circle (2pt) node[above] {$(p^d,0)$};
  \filldraw[fill=blue] (0,1.25) circle (2pt) node[left] {$(0,p^d\nu(Q))$};
  \filldraw[fill=blue] (1,2.8) circle (2pt);
  \filldraw[fill=blue] (2,2.5) circle (2pt);
  \filldraw[fill=blue] (4,1.5) circle (2pt);
  \filldraw[fill=blue] (8,1) circle (2pt);
  \filldraw[fill=blue] (9,2) circle (2pt);
  
\end{tikzpicture}
\caption{In this example, $p\in B_n$ and $1,2,p^d-2,p^d-1\notin B_n$}\label{figuranesmple}
\end{figure}

\section{Proof of Theorem \ref{mainthm}}
\noindent\textit{Proof of Theorem \ref{mainthm}:} Since $\nu(g)=\infty$ there exist finitely many $n\in\N$ for which $\Psi_n\neq \emptyset$. Let $\{n_1,\ldots,n_r\}$ ($n_1<\ldots<n_r$) be the set all the natural numbers for which $\Psi_n$ is a plateau. By Theorem \ref{importate} we have 
\[
d(\Psi_{n_i})\mid d(L/K,v)=p^d.
\] 
Hence, for each $i$, $1\leq i\leq r$, $d(\Psi_{n_i})=p^{d_i}$ for some $d_i\in\N_0$. The numbers $d_1,\ldots, d_r$ are uniquely determined and $d=d_1+\ldots+d_r$. Moreover, since ${\rm rk}(v)=1$ the set $\Psi_{n_i}$ is bounded for $1\leq i<r$. By Remark \ref{rmkdefecnoti}, it follows from \eqref{equationsdefectaug} that
\[
d(\Psi_{n_i})=p^{d_i}>1
\]
and consequently $d_i>0$, $1\leq i<r$. 

For every $i$, $1\leq i<r$, consider the set $B_{n_i}$ constructed in the previous section. Set
\[
I_i:=\{j\in \N_0\mid p^j\in B_{n_i}\}.
\]
By \cite[Theorem 1.1]{NOVSpiv} every element of $B_{n_i}$ is a power of $p$, i.e., $B_{n_i}=p^{I_i}$. If $\Psi_{n_r}$ is bounded, then we also define $I_r$ in the analogous way.

For each $i$ such that $\Psi_{n_i}$ is bounded, by Corollary \ref{corolastablecoeff}, for every $F\in {\rm KP}(\Psi_{n_i})$ take $Q_i\in\Psi_{n_i}$ satisfying \eqref{menainglarge}. For every $Q\in \Psi_{n_i}$ with $\nu(Q)\geq \nu(Q_i)$ we have $p^{I_i}\subseteq L_Q(F)$ (by Corollary \ref{corolastablecoeff}). Observe that $a_{Q D}(F)=1$ (by \cite[Proposition 3.5]{michael}) and $D=p^{d_i}$ (by Theorem \ref{importate}). By Proposition \ref{existenceunuequenusI},
\[
a_{Q0}(F)+\sum_{j\in I_i}a_{Qp^j}(F)Q^{p^j}+Q^{p^{d_i}}
\]
is a limit key polynomial for $\Psi_{n_i}$.

Take $i$, $1\leq i\leq r$, such that $\Psi_{n_i}$ is bounded. Suppose that $I$ is any subset satisfying the conditions \textbf{(iv)} and \textbf{(v)} of Theorem \ref{mainthm}. Since for every $Q$ with large enough value,
\[
F_i:=a_{Q0}(F)+\sum_{j\in I_i}a_{Qp^j}(F)Q^{p^j}+Q^{p^{d_i}}
\]
is a limit key polynomial for $\Psi_n$, we deduce from \textbf{(iv)} that $I\subseteq I_i$ (because $a_{Qp^j}(F_i)=0$ if $j\notin I_i$). On the other hand, by Proposition \ref{existenceunuequenusI} we cannot have $I\subsetneq I_i$. Hence, the set $I_i$ is uniquely determined. This concludes the proof of Theorem \ref{mainthm}.

\section{Defect extensions of degree $p$}

\subsection{The rank one case}
We will proceed with the proof of Proposition \ref{independengrankeone}.

\begin{proof}
Since $(L/K,v)$ is a defect extension of degree $p$ it is immediate. In particular, $\Psi_1$ does not have a maximum and is bounded in $\Gamma$. Since $\nu_{x-b}(g)<\infty=\nu(g)$ the plateau $\Psi_1$ admits a limit key polynomial. Theorem \ref{mainthm} implies that $g$ is a limit key polynomial for $\Psi_1$ (because for any limit key polynomial $F$ for $\Psi_1$ we have $p\leq \deg(F)$).

Assume that \eqref{sitAS} is satisfied. Since ${\rm rk}(v)=1$ we can assume that $\Gamma\subseteq\R$. We set
\begin{equation}\label{equapagamas}
\gamma=\sup\{v(\eta-b)\mid b\in K\}\in \R.
\end{equation}
Then ${\rm dist}(\eta,K)=\gamma^-$. Since the only non-trivial convex subgroup of $\Gamma$ is $\{0\}$, $(L/K,v)$ is independent if and only if $\gamma=0$. 

For each $b\in K$ we have
\[
g=(x-b)^p-(x-b)+g(b).
\]
Hence,
\begin{equation}\label{equajaudaup}
\nu_{x-b}(g)=v(g(b))=p\cdot\nu(x-b).
\end{equation}
Set
\begin{equation}\label{equajaudaup1}
\delta=\sup\{\nu_{x-b}(g)\mid b\in K\}=\sup\{\nu_Q(g)\mid Q\in \Psi_1\}.
\end{equation}
By \eqref{equajaudaup} and \eqref{equajaudaup1} we conclude that $\delta=p\cdot \gamma$. By definition of $I_1$ we have $0\in I_1$ if and only if $\delta=\gamma$ and this is satisfied if and only if $\gamma=0$.

Assume now that \eqref{sitkummer} is satisfied. Denote by $\alpha=\frac{v(p)}{p-1}\in \Gamma$. For any $b\in K$ we have
\begin{equation}\label{equqatikummer}
g=(x-b)^p+pb(x-b)^{p-1}+\ldots+pb^{p-1}(x-b)+(b^p-a).
\end{equation}
By \cite[Proposition 3.7]{KR} we have $\gamma\leq\alpha$. In particular, $\nu_{x-b}(g)=p\cdot\nu(x-b)$ and consequently $\delta=p\cdot\gamma$ for $\gamma$ and $\delta$ as in \eqref{equapagamas} and \eqref{equajaudaup1}. Again ${\rm dist}(\eta,K)=\gamma^-$ and analogously to the Artin-Schreier case, the condition for being independent is satisfied if and only if $\gamma=\alpha$. On the other hand, by \eqref{equqatikummer} the condition $0\in I_1$ is equivalent to
\[
v(p)+\gamma=\delta=p\cdot\gamma
\]
and this is equivalent to $\gamma=\alpha$. This ends the proof of Proposition \ref{independengrankeone}. 
\end{proof}\medskip

In what follows, we present the geometric description, as in Section \ref{geometric}, of each case. In Figures \ref{ASextensions} and \ref{Kummerextensions} below we represent the geometric characterization of situations \eqref{sitAS} and \eqref{sitkummer}, respectively. The blue line represents the Newton polygon $\Delta_{x-b}(g)$ for $\nu(x-b)$ large enough. The red line represents the line $\pi$ connecting $(0,\delta)$ and $(p,0)$. This line has equation $\pi(y)=-\gamma y+\delta$. 

For the Artin-Schreier case, we consider the corresponding points that define $\Delta_{x-b}(g)$:
\[
P_1=(0,p\cdot\nu(x-b)),\ P_2=(1,0) \mbox{ and }P_3=(p,0).
\]
In this case, $\gamma\leq 0$. One can see that $0\in I_1$ (i.e., $P_2$ lies on $\pi$) if and only if $\gamma=0$.
\begin{figure}[H]
\centering
\subfigure[Independent extension ($I_1=\{0\}$)]{
\begin{tikzpicture}[line cap=round,line join=round,>=triangle 45,x=0.8cm,y=0.8cm]
    \draw[->] (-0.1,0) -- (6,0) node[right] {$\Q$};
    \draw[->] (0,-2.0) -- (0,2) node[left] {$\R$};

  \draw [line width=1.pt,color=blue] (0,-1.7)-- (4.5,0);
  \draw [line width=1.pt,color=red] (0,0)-- (4.5,0);

\filldraw[fill=red] (0,0) circle (2pt) node[left] {$(0,p\gamma)$};
  
  \filldraw[fill=blue] (4.5,0) circle (2pt) node[above] {$P_3$};
  \filldraw[fill=blue] (1,0) circle (2pt) node[above] {$P_2$};
  \filldraw[fill=blue] (0,-1.7) circle (2pt) node[left] {$P_1$};

\end{tikzpicture}

}
\subfigure[Dependent extension ($I_1=\emptyset$)]{
\begin{tikzpicture}[line cap=round,line join=round,>=triangle 45,x=0.8cm,y=0.8cm]
    \draw[->] (-0.1,0) -- (6,0) node[right] {$\Q$};
    \draw[->] (0,-2.0) -- (0,2.0) node[left] {$\R$};

  \draw [line width=1.pt,color=blue] (0,-1.7)-- (4.5,0);
  \draw [line width=1.pt,color=red] (0,-1)-- (4.5,0);

\filldraw[fill=red] (0,-1) circle (2pt) node[left] {$(0,p\gamma)$};

  \filldraw[fill=blue] (4.5,0) circle (2pt) node[above] {$P_3$};
  \filldraw[fill=blue] (1,0) circle (2pt) node[above] {$P_2$};
  \filldraw[fill=blue] (0,-1.7) circle (2pt) node[left] {$P_1$};
\end{tikzpicture}
}
\caption{Characterization of dependent and independent Artin-Schreier extensions}\label{ASextensions}
\end{figure}
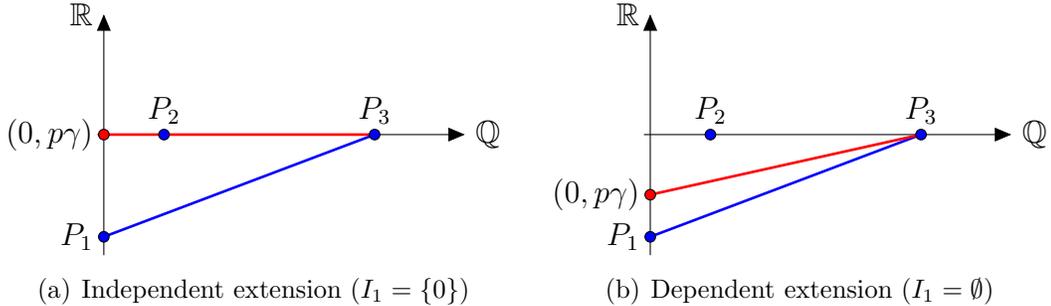

For the Kummer case, we consider the corresponding points that define $\Delta_{x-b}(g)$:
\[
P_1=(0,p\cdot\nu(x-b)),\ P_2=(1,v(p)) \mbox{ and }P_3=(p,0).
\]
In this case, $\gamma\leq \alpha$. One can see that $0\in I_1$ (i.e., $P_2$ lies on $\pi$) if and only if $\gamma=\alpha$.
\begin{figure}[H]
\centering
\subfigure[Independent extension ($I_1=\{0\}$)]{
\begin{tikzpicture}[line cap=round,line join=round,>=triangle 45,x=0.8cm,y=0.8cm]
    \draw[->] (-0.1,0) -- (6,0) node[right] {$\Q$};
    \draw[->] (0,-0.2) -- (0,3.0) node[left] {$\R$};

\filldraw[fill=red] (0,1.9) circle (2pt) node[left] {$(0,p\gamma)$};
    
  \draw [line width=1.pt,color=blue] (0,1)-- (4.5,0);
  \draw [line width=1.pt,color=red] (0,1.9)-- (4.5,0);
  
  \filldraw[fill=blue] (4.5,0) circle (2pt) node[above] {$P_3$};
  \filldraw[fill=blue] (1,1.5) circle (2pt) node[above] {$P_2$};
  \filldraw[fill=blue] (0,1) circle (2pt) node[left] {$P_1$};
\end{tikzpicture}

}
\subfigure[Dependent extension ($I_1=\emptyset$)]{
\begin{tikzpicture}[line cap=round,line join=round,>=triangle 45,x=0.8cm,y=0.8cm]
    \draw[->] (-0.1,0) -- (6,0) node[right] {$\Q$};
    \draw[->] (0,-0.2) -- (0,3.0) node[left] {$\R$};

\filldraw[fill=red] (0,1.5) circle (2pt) node[left] {$(0,p\gamma)$};

  \draw [line width=1.pt,color=blue] (0,0.75)-- (4.5,0);
  \draw [line width=1.pt,color=red] (0,1.5)-- (4.5,0);
  
  \filldraw[fill=blue] (4.5,0) circle (2pt) node[above] {$P_3$};
  \filldraw[fill=blue] (1,1.5) circle (2pt) node[above] {$P_2$};
  \filldraw[fill=blue] (0,0.75) circle (2pt) node[left] {$P_1$};
\end{tikzpicture}
}
\caption{Characterization of dependent and independent Kummer extensions}\label{Kummerextensions}
\end{figure}
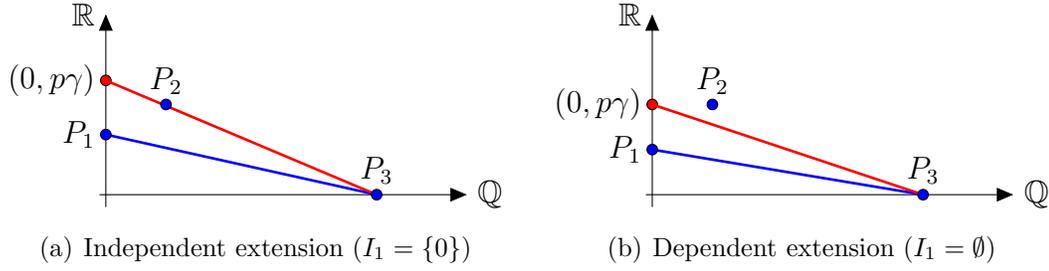

\subsection{The higher rank case}

For both cases, we set $\gamma={\rm dist}(\eta,K)$.

\noindent\textit{Proof of Proposition \ref{propdockuhlma}:} Assume that \eqref{sitAS} is satisfied. As before, for each $b\in K$ we deduce that $\nu(x-b)<0$. Hence,
\begin{equation}\label{equationASextsadacer}
\nu_{x-b}(g)=p\cdot \nu(x-b)<\nu(x-b)
\end{equation}
and consequently $\delta_g\leq 0^-<\infty^-$.  By \cite[Proposition 4.2 and Lemma 2.14]{Kuhl}, $(L/K,v)$ is independent if and only if $p\cdot\gamma=\gamma$. In order to conclude the proof of Proposition \ref{propdockuhlma} for this case it is enough to show that $p\cdot\gamma\neq\gamma$ if and only if $B(g)=\emptyset$. 

It follows from \eqref{equationASextsadacer} that $p\cdot\gamma=\delta_g\leq \gamma$. Since the only possibility for $B(g)$ is $\{1\}$ or $\emptyset$, the condition $B(g)=\emptyset$ is equivalent to the existence of $b\in K$ such that
\[
v(\eta-b)=\nu(x-b)>\delta_g=p\cdot \gamma.
\]
This is, by definition, equivalent to $p\cdot\gamma<\gamma$.

Assume that \eqref{sitkummer} is satisfied and again denote by $\alpha=\frac{v(p)}{p-1}\in \Gamma$. By \cite[Proposition 3.7]{KR} we have
\begin{equation}\label{eqiautnkuhl}
\gamma\leq\alpha+H^-
\end{equation}
for some convex subgroup $H$ of $\Gamma$ that does not contain $v(p)$. Let $H$ be the largest convex subgroup of $\Gamma$ with this property.

By \eqref{eqiautnkuhl} for every $b\in K$ and every $i$, $1\leq i<p$, we have $\nu(x-b)< \frac{v(p)}{p-i}$. In particular,
\[
p\cdot\nu(x-b)<v(p)+i\nu(x-b)\mbox{ for every }i, 1\leq i<p.
\]
Hence, $\nu_{x-b}(g)=p\cdot\nu(x-b)$ and consequently $\delta_g=p\cdot\gamma\leq (p\cdot\alpha)^-<\infty^-$. We also conclude that either $B(g)\ni 1$ or $B(g)=\emptyset$.

For simplicity of notation, we will consider a well-ordered family $\{b_\rho\}_{\rho<\lambda}$ in $K$ such that $\gamma_\rho:=\nu(x-b_\rho)$ form a cofinal family in the lower cut set of $\gamma$.

Suppose that $B(g)=\emptyset$. We will show that there exists $\epsilon\in \Gamma$, $\epsilon>H$ such that $\alpha-\nu(x-c)>\epsilon$ for every $c\in K$. This will imply that
\[
\gamma-\alpha\leq (-\epsilon)^-<H^-
\]
and consequently the extension is dependent. We assume (taking $\gamma_\rho$ large enough) that for every $\rho<\lambda$ we have
\[
v(p)+\gamma_\rho>p\cdot\gamma_\sigma\mbox{ for every }\sigma<\lambda.
\]

If there exist $\rho,\sigma$, $\rho<\sigma<\lambda$ such that $\epsilon_0:=\gamma_\sigma-\gamma_\rho>H$, then for every $\theta$, $\sigma<\theta<\lambda$ we have
\[
\gamma_\theta-\epsilon_0=\gamma_\theta-\gamma_\sigma+ \gamma_\rho>\gamma_\rho>p\cdot\gamma_\theta-v(p).
\]
Hence
\[
\alpha-\gamma_\theta>\frac{\epsilon_0}{p-1}.
\]
Since $\epsilon_0>H$ and $H$ is a convex subgroup of $\Gamma$, we deduce that $\epsilon:=\frac{\epsilon_0}{p-1}>H$.

Suppose that for every $\rho,\sigma$, $\rho<\sigma<\lambda$ we have $\gamma_\sigma-\gamma_\rho\not > H$. Since $H$ is convex this implies that $\gamma_\sigma-\gamma_\rho\in H$. Condition \eqref{eqiautnkuhl} implies that $\alpha-\gamma_\rho>H$ for every $\rho<\lambda$. Fix $\rho<\lambda$ and set $\epsilon=\frac{\alpha-\gamma_\rho}{2}$. For every $\sigma$, $\rho<\sigma<\lambda$, we have
\begin{equation}\label{equationdacertodeno}
\alpha-\gamma_\sigma=\frac{\alpha-\gamma_\sigma}{2}+\frac{\alpha-\gamma_\sigma}{2}>\frac{\alpha-\gamma_\sigma}{2}=\epsilon+\frac{\gamma_\rho-\gamma_\sigma}{2}.
\end{equation}
We claim that $\alpha-\gamma_\sigma>\epsilon$. Indeed, if this were not the case, then by \eqref{equationdacertodeno} we would have
\[
0\leq \epsilon-\alpha+\gamma_\sigma<\frac{\gamma_\sigma-\gamma_\rho}{2}.
\]
Since $\frac{\gamma_\sigma-\gamma_\rho}{2}\in H$ (and $H$ is convex) this would imply that $\epsilon-\alpha+\gamma_\sigma\in H$. On the other hand, we have
\[
\epsilon-\alpha+\gamma_\sigma=\frac{\alpha-\gamma_\rho}{2}-\alpha+\gamma_\sigma=\frac{(\gamma_\sigma-\gamma_\rho)}{2}-\frac{(\alpha-\gamma_\sigma)}{2}.
\]
We would obtain that $\alpha-\gamma_\sigma\in H$ and this is a contradiction to \eqref{eqiautnkuhl}.

For the converse, assume that $(L/K,v)$ is dependent. Then there exists $\epsilon>H$ such that
\[
\gamma-\alpha\leq\left(-\frac{\epsilon}{p-1}\right)^-<H^-.
\]
This implies that for every $\rho<\lambda$ we have
\[
(p-1)\cdot\gamma_\rho-v(p)<-\epsilon.
\]
Hence,
\begin{equation}\label{equactonfinal}
v(p)+\gamma_\rho>p\cdot \gamma_\rho+\epsilon.
\end{equation}

Let $\Gamma_1$ be the smallest convex subgroup of $\Gamma$ for which \eqref{eqiautnkuhl} is not satisfied for $H$ replaced by $\Gamma_1$. In particular, $\Gamma_1/H$ has rank one, $\epsilon\in \Gamma_1\setminus H$ and
\[
v(p)-(p-1)\cdot\gamma_\rho\in \Gamma_1\setminus H\mbox{ for }\rho\mbox{ large enough.}
\] 
Taking infimum in $\Gamma_1/H$, we deduce that there exists $\rho<\lambda$ such that for every $\sigma$, $\rho<\sigma<\lambda$ we have
\[
p\cdot(\gamma_\sigma-\gamma_\rho)<\epsilon.
\]
This and \eqref{equactonfinal} imply that
\[
v(p)+\gamma_\rho>p\cdot \gamma_\sigma\mbox{ for every }\sigma, \rho<\sigma<\lambda.
\]
Hence $1\notin B(g)$ and consequently $B(g)=\emptyset$. This concludes the proof of Proposition \ref{propdockuhlma}.


\begin{thebibliography}{99}
\bibitem{Netal}M. Alberich-Carrami$\tilde{\mbox{n}}$ana, Alberto F. Boix, J. Fern\'andez, J. Gu\`ardia, E. Nart, J. Ro\'e, \emph{Of limit key polynomials}, Illin. J. Math. {\bf 65}, No.1 (2021), 201--229, DOI: 10.1215/00192082-8827671.

\bibitem{Kap} I. Kaplansky, \textit{Maximal fields with valuations I}, Duke Math. Journ. \textbf{9} (1942), 303--321, DOI: 10.1215/S0012-7094-42-00922-0.

\bibitem{Kuhl} F.-V. Kuhlmann, \textit{A classification of Artin-Schreier defect extensions and characterizations of defectless fields}, Illinois J. Math. \textbf{54} (2) (2010), 397--448, DOI: 10.1215/ijm/1318598666.

\bibitem{KR} F.-V. Kuhlmann and A. Rzepka, \textit{The valuation theory of deeply ramified fields and its connection with defect extensions}, Trans. Amer. Math. Soc. \textbf{376} (2023), 2693--2738, DOI: 10.1090/tran/8790.

\bibitem{michael} M. Moraes and J. Novacoski, \textit{Limit key polynomials as $p$-polynomials}, J. Algebra \textbf{579} (2021), 152--173, DOI: 10.1016/j.jalgebra.2021.03.024.

\bibitem{NN} E. Nart and J. Novacoski, \textit{The defect formula}, Adv. Math. \textbf{428} (2023), 109153, DOI: 10.1016/j.aim.2023.109153.

\bibitem{NovaMP} J. Novacoski, \textit{Key polynomials and minimal pairs}, J. Algebra \textbf{523} (2019), 1--14, DOI: 10.1016/j.jalgebra.2018.12.022.

\bibitem{Nov1} J. Novacoski, \textit{On MacLane-Vaqui\'e key polynomials}, Journal of Pure and Applied Algebra Volume 225, Issue 8 (2021), 106644, DOI: 10.1016/j.jpaa.2020.106644.

\bibitem{SopivNova} J. Novacoski and M. Spivakovsky, \textit{Key polynomials and pseudo-convergent sequences}, J. Algebra \textbf{495} (2018), 199--219, DOI: 10.1016/j.jalgebra.2017.11.006.

\bibitem{NOVSpiv} J. Novacoski and M. Spivakovsky, \textit{On stable and fixed polynomials}, J. Pure Appl. Algebra Vol. \textbf{227} Issue \textbf{3} (2023), 107216, DOI: 10.1016/j.jpaa.2022.107216.
\end{thebibliography}
\end{document}